\documentclass{siamart190516}
\usepackage[utf8]{inputenc}
\usepackage{amsmath}
\usepackage{amsfonts}
\usepackage{amssymb}
\usepackage{graphicx}
\graphicspath{{./}
}
\usepackage{subcaption}

\usepackage{xcolor}
\definecolor{codegray}{rgb}{0.5,0.5,0.5}
\definecolor{backcolour}{rgb}{0.95,0.95,0.92}
\usepackage{listings}
\lstdefinestyle{mystyle}{
    backgroundcolor=\color{backcolour},   
    numberstyle=\tiny\color{codegray},
    stringstyle=\color{codepurple},
    basicstyle=\ttfamily\footnotesize,
    breakatwhitespace=false,         
    breaklines=true,                 
    captionpos=b,                    
    keepspaces=true,                 
    numbers=left,                    
    numbersep=5pt,                  
    showspaces=false,                
    showstringspaces=false,
    showtabs=false,                  
    tabsize=2
}

\usepackage{tikz}
\tikzset{every picture/.style={line width=0.11mm}}
\newcommand{\oPerpSymbol}{\begin{tikzpicture}[scale=0.134] %% scale=0.134 for 12pt; scale=0.112 for 10pt
  \draw (0,-0.5)--(0,1); \draw (-0.866,-0.5)--(0.866,-0.5);
  \draw (0,0) circle [radius=1];
\end{tikzpicture}}
\newcommand{\obot}{\mathbin{\raisebox{-1pt}{\oPerpSymbol}}}

\usepackage[colorinlistoftodos]{todonotes}
\reversemarginpar

\usepackage{tikz}
\usetikzlibrary{cd}
\newsiamremark{remark}{Remark}

\usepackage[
backend=biber,
style=numeric,
sorting=nty,
url=false,
isbn=false,
doi=false
]{biblatex}
\usepackage{hyperref}
\makeatletter
\renewcommand*{\eqref}[1]{%
  \hyperref[{#1}]{\textup{\tagform@{\ref*{#1}}}}%
}
\makeatother

\title{Numerical solution of the div-curl problem by finite element exterior calculus.}
\author{Pascal AZERAD 
\and
Marien-Lorenzo HANOT. \thanks{ IMAG, 
 Université de Montpellier,  CNRS, Montpellier, FRANCE,
 \email{pascal.azerad@umontpellier.fr},
\email{marien-lorenzo.hanot@umontpellier.fr}.
}}

\begin{document}
\maketitle

\begin{abstract} We are interested in the numerical reconstruction of a vector field with  prescribed divergence and curl in a general domain of $\mathbb{R}^3$ or $\mathbb{R}^2$, not necessarily contractible. To this aim,
we introduce some basic concepts of  finite element exterieur calculus \cite{feec-cbms} and  rely heavily on recent results of P. Leopardi and A. Stern \cite{Leopardi2016}.
The goal  of the paper  is to take advantage of  the links between usual vector calculus and exterior calculus and  show the interest of the exterior calculus framework, without too much prior knowledge of the subject.
We start by  describing the method used for contractible   domains and its implementation using the FEniCS library (see \url{fenicsproject.org}).
We then address the problems encountered with non contractible domains and general boundary conditions and explain how to adapt the method to handle these cases.
Finally we give some numerical results obtained with this method, in dimension 2 and 3.
\end{abstract}


\section{Outline of the paper}
In electromagnetism or fluid mechanics, one often encounter the problem of reconstructing a vector field with prescribed divergence and curl. 
As is well-known, the Biot-Savart law allows the reconstruction of a solenoidal vector field $u$ in  the whole space $\mathbb{R}^d$ from its  curl.
$$u (x,t) = \int_{\mathbb{R}^d} K(x,y) \times \mathrm{curl} \,u (y,t) \,dy $$
with 
$$K(x,y) = \left\{
  \begin{array}{ll}
\frac{1}{4\pi}\frac{x-y}{\vert x-y \vert^3} & \mathrm{if}\, d=3 \\
 \frac{1}{2\pi}\frac{x-y}{\vert x-y \vert^2} & \mathrm{if}\, d=2. 
 \end{array}
\right.$$
However it is a singular integral in the whole space and is not easily converted to a numerical algorithm, nor suitable in a bounded domain. 
The div-curl problem \eqref{eq:divcurlpb} has been addressed theoretically e.g. in \cite{daco} or \cite{Auchmuty-Alexander}. The \emph{numerical} computation of the solution in  a bounded domain is less documented, see e.g. \cite{vallisimple} for a finite element solution using divergence free elements and an coercive variational form. The purpose of this paper is to show that the framework of differential forms and exterior calculus greatly simplifies  both the theory and the finite element solution.  In section \ref{Helmholtzdecomposition}  we  show how the div-curl problem is related to the classical Helmholtz decomposition.
In section \ref{Exteriorcalculus} we introduce the main tools of exterior calculus.
In section \ref{weakformulation} we give a natural weak formulation of the div-curl problem,  introduced in \cite{Leopardi2016}, which is well posed. 
In section \ref{Finiteelements} we detail the mixed finite elements compatible with the weak formulation. The implementation within the unified form language \cite{Logg-Mardal-Wells}  is sketched in section \ref{implem}.
The particular case of the 2D case is detailed in section \ref{2D}. 
The case when the domain is not contractible is addressed in section \ref{Noncontractible}.
In section \ref{Boundaryconditions} we describe the cases of natural and essential boundary conditions. 
We close the paper with numerical tests  in 2 and 3 dimensions in section \ref{numerical}.
\section{Helmholtz decomposition} \label{Helmholtzdecomposition}
Let $\Omega$  be a bounded domain in $\mathbb{R}^d$, the div-curl problem consists in finding a vector field with a prescribed divergence and curl. 
For $g$ a scalar field and $f$ a vector field we seek $u$ such that
\begin{equation} \label{eq:divcurlpb}
\begin{aligned}
\nabla \cdot u &= g \quad \mathrm{in} \; \Omega,\\
\nabla \times u &= f\quad \mathrm{in} \; \Omega.
\end{aligned}
\end{equation}
Of course, one must add boundary conditions and specify the regularity of these fields.
The existence of a solution is not guaranteed, indeed some vector fields are not written as the curl of other vector fields. The uniqueness is also an issue, indeed   depending on the domain topology  there may exist  non trivial fields with vanishing divergence and curl, the so called harmonic fields. 
Theses problem make the elaboration of a stable scheme for the numerical solution quite complicated, because one must make sure that the fields $f$ and $g$ are  compatible and in some way  filter  out the  harmonic fields.

Problem \eqref{eq:divcurlpb} is very close to the Helmholtz decomposition.
The Helmholtz decomposition is central in vector calculus.  In its classical formulation 
it states that any field of $\mathbb{R}^3$ which is sufficiently smooth and decreases sufficiently fast at infinity can be decomposed into the sum of a gradient and a curl.
The problem is then to compute this decomposition.
For a given vector field $\mathbf{F}$, find a vector potential $\mathbf{A}$ and a scalar potential $\phi$ such that
\begin{equation} \label{eq:helmholtz}
\mathbf{F} = - \nabla \phi + \nabla \times \mathbf{A}.
\end{equation}
Although this is not equivalent to the equations \eqref{eq:divcurlpb}, we can see the relation by taking $u = \mathbf{A}$ in \eqref{eq:helmholtz} and $f = \mathbf{F}$.

Thus, the problem of the absence of compatibility of $f$ in \eqref{eq:divcurlpb} can be solved by looking simultaneously for  a vector potential $u$ and
a scalar potential $\phi$ as described by \eqref{eq:helmholtz}.
The general idea is then to couple these problems to obtain a well-posed system.

The classical Helmholtz decomposition \eqref{eq:helmholtz} assumes that functions are smooth in  the whole space. For bounded domains, we have the following standard result,  see e.g. \cite{Girault-Raviart} or chapter 9 of \cite{Dautray-Lions}.
Let $\Omega$ be a bounded, simply-connected, Lipschitz domain in $\mathbb{R}^d$,  for any $\mathbf{F} \in (L^2(\Omega))^d$
there exists  $\phi \in H^1(\Omega)$ and $\mathbf{A} \in H(\text{curl},\Omega)$ such that  the decomposition \eqref{eq:helmholtz} is valid.  Of course, there is no uniqueness, since one can add constants to $\phi$ and any gradient field to $\mathbf{A}$.
Prescribing boundary conditions, we get the $L^2$-orthogonal decompositions, proved e.g. in Arnold \cite{feec-cbms}:
\begin{align}
(L^2(\Omega))^d & =   \nabla (H^1(\Omega))  \obot \nabla \times (H_0(\text{curl},\Omega)) \label{eq:helmholtz-mag} \\
(L^2(\Omega))^d &= \nabla (H^1_0(\Omega))  \obot  \nabla \times (H(\text{curl},\Omega)). \label{eq:helmholtz-elec}
\end{align}
We recall the standard definitions: 
$H^1_0(\Omega) = \{u \in L^2(\Omega); \; \nabla u \in L^2(\Omega)^d, \; u  = 0 \;  \text{on} \; \partial\Omega\}$,
$H_0(\text{curl},\Omega)) = \{u \in L^2(\Omega)^d; \; \text{curl}(u) \in L^2(\Omega)^d, \; u\times n = 0  \;\text{on} \;\partial\Omega\}$,
$n$ being the unit outer normal to the boundary $\partial \Omega$.

\section{Exterior calculus} \label{Exteriorcalculus}
Exterior calculus is based on differential forms, which are applications that at any point of the domain associate an alternating multilinear form.
In particular the $0$-forms are simply functions and the $1$-forms can be seen as vector fields (by identifying  linear forms with vectors, thanks to the usual inner product ).
These spaces are endowed  with a natural inner product (built from the one on vectors), which  can be integrated on the whole domain.
This allows us to define Hilbert spaces, in particular we will note $L^2 \Lambda^k(\Omega)$ the set of $L^2$-integrable $k$-forms.
As alternating multilinear applications, it is possible to define an operation between a $k$-form $\alpha$ and an $l$-form $\beta$ denoted by the wedge product $\wedge$ 
giving a $k+l$-form $\alpha \wedge \beta = - \beta \wedge \alpha$. This operation is pointwise and allows to define a basis of the whole algebra from a basis of $1$-forms.
In dimension $3$ a local basis is given in the table \ref{table:basisdim3}.
\begin{table}[h!]
\centering
\begin{tabular}{c|c|c|c|c} Space&$\Lambda^0(\Omega)$ & $\Lambda^1(\Omega)$ & $\Lambda^2(\Omega)$ & $\Lambda^3(\Omega)$ \\
\hline
Basis& $1$ & $ dx, dy, dz $ & $ dy \wedge dz, -dx \wedge dz, dx \wedge dy$ & $ dx \wedge dy \wedge dz$
\end{tabular}
\caption{Basis of the exterior algebra on a domain of $\mathbb{R}^3$.} \label{table:basisdim3}
\end{table}

These forms are also equipped with the exterior derivative operator $d$.
This operator acts globally on the exterior algebra $d : \bigoplus\limits_{k = 0} H \Lambda^k(\Omega) \rightarrow \bigoplus\limits_{k = 0} H \Lambda^k(\Omega)$ 
but it is often interesting to look at its action restricted to each space, where it will respect a notion of degree, specifically :
\[d: H\Lambda^k(\Omega) \rightarrow H\Lambda^{k+1}(\Omega)\ .\]
The operator $d$ is defined as the usual differential on the $0$-forms,
by the property $d \circ d = 0$ and extended on the other forms degree by $d(\alpha \wedge \beta) = d\alpha \wedge \beta + (-1)^k \alpha \wedge d \beta$ for a $k$-form $\alpha$.
This differential operator is of course not defined on all $L^2$-forms but only on a dense subset.
As for the Sobolev spaces, we then consider the subset of $L^2$-forms such that their exterior derivative is again $L^2$.
We note this set $H \Lambda(\Omega)$.

One can see a great similarity with vector calculus, for instance $d \circ d = 0$ corresponds to the well known identity $\mathrm{curl} ( \mathrm{grad}) = 0$ and $\mathrm{div} (\mathrm{curl}) = 0$.
This similarity is deep  since there is a natural identification between vector fields and differential  forms, the identification commutating with these differential operators.
The  identification is depicted by  diagrams \ref{cd:dim2} for dimension $2$ and \ref{cd:dim3} for dimension $3$. We use the natural shortcut $dx \to x, dy \to  y, dz \to z$ to mean that 
the one form $a\,dx + b\,dy + c\,dz$ is identified with the vector $a\cdot\mathbf{e_x} + b \cdot \mathbf{e_y}  + c \cdot \mathbf{e_z} = (a,b,c)$. In the same way, for  2-forms  $dy\wedge dz \to x$,  $ dz\wedge dx \to y $, $dx\wedge dy \to z $ means that $a \, dy\wedge dz +b\,  dz\wedge dx + c \, dx \wedge dy$  is  identified with the vector $ (a,b,c)$, whereas for 3-forms  $dx\wedge dy\wedge dz \to 1$ simply means that the 3 form $ a \,dx\wedge dy\wedge dz $ is identified with the scalar $a$. In dimension 2,  the same conventions apply.
We can notice that there are two possible identifications in dimension 2.
We will refer to the first one in which $\nabla \times$ appears as the curl identification and to the other one (in which $\nabla \cdot$ appears) as the divergence identification.
We refer again to \cite{feec-cbms} for  a thorough exposition of this so called "proxy" identifications.

\begin{figure}
\centering
\begin{tikzpicture}[scale=1.75] 
%% Place nodes
\node (h1) at (0,0) {$H^1(\Omega)$};
\node (h2) at (2,0) {$H(\text{curl}, \Omega)$};
\node (h3) at (4,0) {$H(\text{div}, \Omega)$};
\node (h4) at (6,0) {$L^2(\Omega)$};
\node (l0) at (0,1.5) {$H\Lambda^0(\Omega)$};
\node (l1) at (2,1.5) {$H\Lambda^1(\Omega)$};
\node (l2) at (4,1.5) {$H\Lambda^2(\Omega)$};
\node (l3) at (6,1.5) {$H\Lambda^3(\Omega)$};
%% subnode
\node (h1b) at (0,0.25) {$1$};
\node (h2b1) at (1.75,0.25) {$x$};
\node (h2b2) at (2,0.25) {$y$};
\node (h2b3) at (2.25,0.25) {$z$};
\node (h3b1) at (3.75,0.25) {$x$};
\node (h3b2) at (4,0.25) {$y$};
\node (h3b3) at (4.25,0.25) {$z$};
\node (h4b) at (6,0.25) {$1$};
\node (l0b) at (0,1.25) {$1$};
\node (l1b1) at (1.75,1.25) {$dx$};
\node (l1b2) at (2,1.25) {$dy$};
\node (l1b3) at (2.25,1.25) {$dz$};
\node (l2b1) at (3.5,1.25) {$dy \wedge dz$};
\node (l2b2) at (4,1) {$-dx \wedge dz$};
\node (l2b3) at (4.5,1.25) {$dx \wedge dy$};
\node (l3b) at (6,1.25) {$dx \wedge dy \wedge dz$};
%% draw arrow
\draw[->,blue,thick]
(h1) edge node[auto] {$\nabla$} (h2)
(h2) edge node[auto] {$\nabla \times $} (h3)
(h3) edge node[auto] {$\nabla \cdot $} (h4)
(l0) edge node[auto] {$d$} (l1)
(l1) edge node[auto] {$d$} (l2)
(l2) edge node[auto] {$d$} (l3);
\draw[-,red]
(h1b) edge (l0b)
(h2b1) edge (l1b1)
(h2b2) edge (l1b2)
(h2b3) edge (l1b3)
(h3b1) edge (l2b1)
(h3b2) edge (l2b2)
(h3b3) edge (l2b3)
(h4b) edge (l3b);
\end{tikzpicture}
\caption{Identification between vectors and forms on a domain $\Omega \subset \mathbb{R}^3$.}
\label{cd:dim3}
\end{figure}

\begin{figure}
\centering
\begin{tikzpicture}[scale=1.75] 
%% Place nodes
\node (h1) at (0,0) {$H^1(\Omega')$};
\node (h2) at (2,0) {$H(\text{curl}, \Omega')$};
\node (h4) at (4,0) {$L^2(\Omega')$};
\node (l0) at (0,1.5) {$H\Lambda^0(\Omega')$};
\node (l1) at (2,1.5) {$H\Lambda^1(\Omega')$};
\node (l2) at (4,1.5) {$H\Lambda^2(\Omega')$};
%% subnode
\node (h1b) at (0,0.25) {$1$};
\node (h2b1) at (1.85,0.25) {$x$};
\node (h2b2) at (2.15,0.25) {$y$};
\node (h4b) at (4,0.25) {$1$};
\node (l0b) at (0,1.25) {$1$};
\node (l1b1) at (1.85,1.25) {$dx$};
\node (l1b2) at (2.15,1.25) {$dy$};
\node (l2b3) at (4,1.25) {$dx \wedge dy$};
%% draw arrow
\draw[->,blue,thick]
(h1) edge node[auto] {$\nabla$} (h2)
(h2) edge node[auto] {$\nabla \times $} (h4)
(l0) edge node[auto] {$d$} (l1)
(l1) edge node[auto] {$d$} (l2);
\draw[-,red]
(h1b) edge (l0b)
(h2b1) edge (l1b1)
(h2b2) edge (l1b2)
(h4b) edge (l2b3);
\end{tikzpicture}
\centering
\begin{tikzpicture}[scale=1.75]
%% Place nodes
\node (h1) at (0,0) {$H^1(\Omega')$};
\node (h2) at (2,0) {$H(\text{div}, \Omega')$};
\node (h4) at (4,0) {$L^2(\Omega')$};
\node (l0) at (0,1.5) {$H\Lambda^0(\Omega')$};
\node (l1) at (2,1.5) {$H\Lambda^1(\Omega')$};
\node (l2) at (4,1.5) {$H\Lambda^2(\Omega')$};
%% subnode
\node (h1b) at (0,0.25) {$1$};
\node (h2b1) at (1.85,0.25) {$x$};
\node (h2b2) at (2.15,0.25) {$y$};
\node (h4b) at (4,0.25) {$1$};
\node (l0b) at (0,1.25) {$1$};
\node (l1b1) at (2.2,1.25) {$dy$};
\node (l1b2) at (1.8,1.25) {$-dx$};
\node (l2b3) at (4,1.25) {$dx \wedge dy$};
%% draw arrow
\draw[->,blue,thick]
(h1) edge node[auto] {$\nabla^\perp$} (h2)
(h2) edge node[auto] {$\nabla \cdot $} (h4)
(l0) edge node[auto] {$d$} (l1)
(l1) edge node[auto] {$d$} (l2);
\draw[-,red]
(h1b) edge (l0b)
(h2b1) edge (l1b1)
(h2b2) edge (l1b2)
(h4b) edge (l2b3);
\end{tikzpicture}
\caption{Two possibles identification between vectors and forms on a domain $\Omega' \subset \mathbb{R}^2$.}
\label{cd:dim2}
\end{figure}

Advantages of using  exterior calculus  are twofold.
First a lot of work has already been done for the discretization of these spaces and  operator $d$  (see \cite{feec-cbms}, \cite{Arnold_2010}) as we will see in Section \ref{Finiteelements}.
Second is that  the previously described vector -  differential form  identifications makes the nature of our operators clearer.
Instead of having three different operators $\nabla$, $\nabla \cdot$ and $\nabla \times$ we see that they become after identification a single operator $d$ applied to different spaces.
This allows us to unify the different types of "Helmholtz decomposition" (called Hodge decomposition in the context of exterior calculus) into a single one 
and allows us to understand how to find  stable formulations of the problems \eqref{eq:divcurlpb} and \eqref{eq:helmholtz}.

Before stating this decomposition, let us introduce the adjoints of the operator $d$. 
Indeed, as we can see in the diagram \ref{cd:dim3}, after identification the image of $\nabla$ and that of $\nabla \times$ do not belong to the same space.
This is perfectly normal because these two operators will not occur in the same decomposition, $\nabla$ will occur in the decomposition of $1$-forms and $\nabla \times$ in that of $2$-forms.
The operators completing these decompositions are the adjoints of these operators.

We define the codifferential operator $\delta$ as the adjoint operator of $d$. 
Since it is not defined on the whole $L^2 \Lambda(\Omega)$, we note its domain $\dot{H}^\star \Lambda(\Omega)$. 
Explicitly for $\Omega$ a bounded domain of $\mathbb{R}^3$, the adjoint of $(\nabla, H^1(\Omega))$ is $(- \nabla \cdot, H_0 (\text{div},\Omega))$, 
the adjoint of $( \nabla \times, H(\text{curl},\Omega))$ is $(\nabla \times, H_0 (\text{curl}, \Omega))$ and the adjoint of $(\nabla \cdot, H(\text{div},\Omega))$ is $(- \nabla, H^1_0 (\Omega))$, \cite{feec-cbms} for details.
Thus under the identification in fig. \ref{cd:dim3} we have 
\begin{equation} \label{eq:identadj}
\begin{aligned}
(- \nabla \cdot, H_0 (\text{div},\Omega)) &= (\delta, \dot{H}^\star \Lambda^1(\Omega)), \\
(\nabla \times, H_0 (\text{curl},\Omega)) &= (\delta, \dot{H}^\star \Lambda^2(\Omega)), \\
(- \nabla, H_0^1 (\Omega)) &= (\delta, \dot{H}^\star \Lambda^3(\Omega)).
\end{aligned}
\end{equation}

Let $\Omega$ is a bounded, contractible, Lipschitz domain, to the two  Helmholtz decompositions \eqref{eq:helmholtz-mag}-\eqref{eq:helmholtz-elec}
 correspond the two  following  Hodge decompositions, depending on the choice to identify $(L^2(\Omega))^3$ with $L^2 \Lambda^1(\Omega)$ or with $L^2 \Lambda^2 (\Omega)$.
They are respectively given by
\begin{equation} \label{eq:hodge1f}
L^2 \Lambda^1 (\Omega) = d (H \Lambda^0(\Omega)) \obot \delta (\dot{H}^\star \Lambda^2(\Omega)),
\end{equation}
\begin{equation} \label{eq:hodge2f}
L^2 \Lambda^2 (\Omega) = \delta (\dot{H}^\star \Lambda^3(\Omega)) \obot d (H \Lambda^1(\Omega)).
\end{equation}

Considering $\Lambda^4 (\Omega) \equiv \lbrace 0 \rbrace$ we can also extend the decomposition to the $3$-forms:
\[ L^2 \Lambda^3 (\Omega) = d (H \Lambda^2(\Omega)). \]
This reflects the fact that $\nabla \cdot (H(\text{div},\Omega)) = L^2(\Omega)$.

Finally, if we try to extend the formula to $0$-forms we will encounter a problem.
Indeed, the constant functions are not in the range of $(\nabla \cdot, H_0(\text{div},\Omega))$.
Moreover the constant functions are exactly the kernel of $\nabla$ and are orthogonal to the range of $(\nabla \cdot, H_0(\text{div},\Omega))$,
which is a generic fact that we will develop in  Section \ref{Noncontractible} generalizing the algorithm to general domains.
We will call such function harmonic ($0$-)forms and note their set $\mathfrak{H}^0 \subset L^2 \Lambda^0(\Omega)$ (the set of constant $0$-forms).
The Hodge decomposition is then
\[ L^2 \Lambda^0 (\Omega) = \delta (\dot{H}^\star \Lambda^1 (\Omega)) \obot \mathfrak{H}^0. \]

We can gather these four decompositions into a decomposition of the total space $L^2 \Lambda (\Omega) = \bigoplus\limits_{k = 0}^3 L^2 \Lambda^k(\Omega)$
and on this domain, the Hodge decomposition is given by
\begin{equation} \label{eq:hodge}
L^2 \Lambda (\Omega) = d(H \Lambda(\Omega)) \obot \delta( \dot{H}^\star \Lambda (\Omega)) \obot \mathfrak{H}^0.
\end{equation}

%\begin{remark}
%Being contractible is stronger than just being simply-connected, 
%so \eqref{eq:hodge} is more restrictive than the statement of the standard Helmholtz decomposition \eqref{eq:helmholtz}.
%This difference comes from the fact that the Hodge decomposition \eqref{eq:hodge1f} and \eqref{eq:hodge2f} use respectively the spaces 
%$H_0(\text{curl},\Omega)$ instead of $H(\text{curl},\Omega)$ and $H^1_0(\Omega)$ instead of $H^1(\Omega)$.
%\end{remark}

\section{Weak formulation}\label{weakformulation}
The Hodge decomposition \eqref{eq:hodge} gives the quasi invertibility of the operator $d + \delta : H \Lambda( \Omega) \cap \dot{H}^\star \Lambda (\Omega) \rightarrow L^2 \Lambda (\Omega)$. We mean  by that  the following statement.
\begin{theorem}
 Let $f \in L^2 \Lambda (\Omega) $, there exists $w \in H \Lambda( \Omega) \cap \dot{H}^\star \Lambda (\Omega) $ and $ h \in \mathfrak{H}^0$ such that
$f = (d+\delta) w +h$.
\end{theorem}
\begin{proof}
Let $f$ be in  $L^2 \Lambda (\Omega)$, from the Hodge decomposition \ref{eq:hodge}, there exist $u \in H \Lambda( \Omega) $,  $v \in \dot{H}^\star \Lambda (\Omega)$ and $ h \in \mathfrak{H}^0$ such that 
$$f = du + \delta v +h. $$
Let us apply the Hodge decomposition to $u$ and $v$. We can write
$u = dm + \delta p + q$ and $v = dr + \delta s + t$,  where $m,r$ belong to $H \Lambda( \Omega) $,  $p,s$ belong to $\dot{H}^\star \Lambda (\Omega)$ and $ q,t $ are in $\mathfrak{H}^0$.
Now take $w = \delta p + d r$. Using $d^2= \delta^2 =0$ and the fact that $q$ and $t$ are harmonic forms, we  can compute 
$$(d+\delta) w = d \delta p + \delta d r= d( dm + \delta p + q) + \delta (dr + \delta s + t)= du + \delta v.$$
So we have proved that $f = (d+\delta) w +h.$
Furthermore we have that $\delta p \in  L^2 \Lambda (\Omega)$  and $d m \in  L^2 \Lambda (\Omega)$, thus $w  \in  L^2 \Lambda (\Omega).$
Now $dw = d\delta p = du$  belongs to   $L^2 \Lambda (\Omega)$ and $\delta w =\delta d r = \delta v$  belongs to $L^2 \Lambda (\Omega)$, hence
$w \in H \Lambda( \Omega) \cap \dot{H}^\star \Lambda (\Omega).$
\end{proof}
This operator is called the Hodge-Dirac operator and is thoroughly studied in \cite{Leopardi2016}.
The non-invertibility of this operator comes from the space of harmonic forms $\mathfrak{H}= \mathfrak{H}^0$, where we omit the superscript from now on,  being both in the kernel of the operator and orthogonal to its range.
This can be circumvented by using a space orthogonal to the harmonic forms.

Hence the primal formulation of  \eqref{eq:divcurlpb} becomes:\\
Given $f \in L^2 \Lambda (\Omega) \cap \mathfrak{H}^\perp$, find $u \in H \Lambda( \Omega) \cap \dot{H}^\star \Lambda (\Omega) \cap \mathfrak{H}^\perp$
such that $\forall v \in H \Lambda (\Omega)$,
\begin{equation} \label{eq:weak1}
\langle d u, v \rangle + \langle \delta u, v \rangle = \langle f, v \rangle.
\end{equation}
By orthogonality arguments we can see that testing with $v \perp \mathfrak{H}$ does not matter.
Since $d$ is adjoint to $\delta$ we can see that \eqref{eq:weak1} is equivalent to:\\
Given $f \in L^2 \Lambda (\Omega) \cap \mathfrak{H}^\perp$, find $u \in H \Lambda( \Omega) \cap \mathfrak{H}^\perp$ such that $\forall v \in H \Lambda (\Omega)$,
\begin{equation} \label{eq:weak2}
\langle d u, v \rangle + \langle u, d v \rangle = \langle f, v \rangle.
\end{equation}

We remove the condition of orthogonality $f \in  \mathfrak{H}^\perp$ by introducing a new pair of variables.
The problem becomes:\\
Given $f \in L^2 \Lambda (\Omega)$, find $u \in H \Lambda( \Omega), p \in \mathfrak{H}$ such that $\forall v \in H \Lambda (\Omega), \forall q \in \mathfrak{H}$,
\begin{align}
\langle d u, v \rangle + \langle u, d v \rangle + \langle p, v \rangle &= \langle f, v \rangle, \label{eq:weak31} \\
\langle u, q \rangle &= 0. \label{eq:weak32}
\end{align}
We have added the component $\langle p, v \rangle$ to the equation \eqref{eq:weak31}, since by orthogonality we must have $p= P_\mathfrak{H} f$ (the orthogonal projection of $f$ on $\mathfrak{H}$)
and so we will effectively solve for $(d + \delta)u = f - P_{\mathfrak{H}} f$.
Equation \eqref{eq:weak32} ensures injectivity by imposing $P_\mathfrak{H} u = 0$.
\begin{theorem}
For any $f \in L^2 \Lambda (\Omega)$, there is a unique $(u,p) \in H \Lambda( \Omega) \times \mathfrak{H}$ 
solution of \eqref{eq:weak31}-\eqref{eq:weak32}. 
Moreover there exists $c > 0$ depending only on $\Omega$ such that $\Vert u \Vert + \Vert d u \Vert + \Vert p \Vert \leq c \Vert f \Vert$.
\end{theorem}
For the sake of completeness we give the proof, which follows closely the one given by Stern, \cite[Theorem~6]{Leopardi2016}.
\begin{proof}
Let $u \in H \Lambda (\Omega)$ and $p \in \mathfrak{H}$, 
from the Hodge decomposition \eqref{eq:hodge}, there exist $m \in H \Lambda(\Omega)$, $n \in \dot{H}^\star \Lambda (\Omega)$ and $o \in \mathfrak{H}$ such that 
$u = dm + \delta n + o$. 
Applying once again the Hodge decomposition to $m$, there exist $r \in H \Lambda(\Omega)$, $s \in \dot{H}^\star \Lambda (\Omega)$ and $t \in \mathfrak{H}$ such that 
$m = dr + \delta s + t$.
Take $v = d \delta n + \delta s + p$ and $q = o$ in \eqref{eq:weak31}-\eqref{eq:weak32}.
Noticing that $d v = d \delta s = d m$, $d u = d \delta n$, using the Poincaré inequality $\Vert \delta n \Vert \leq c_p \Vert d \delta n \Vert$ 
and using the orthogonality of the Hodge decomposition the equation reads:
\begin{equation}
\begin{aligned}
\langle du , v \rangle + \langle u, dv \rangle + \langle p, v \rangle + \langle u, q \rangle &=
\langle d \delta n , d \delta n \rangle + \langle u, d m \rangle + \langle p, p \rangle + \langle u, o \rangle \\
&= \Vert d \delta n \Vert^2 + \Vert d m\Vert^2 + \Vert p \Vert^2 + \Vert o\Vert \\
&\geq \frac{1}{2} \Vert d \delta n \Vert^2 + \frac{1}{2 c_p^2} \Vert \delta n \Vert^2 + \Vert d m \Vert^2 + \Vert p \Vert^2 + \Vert o\Vert \\
&\geq \frac{1}{2} \Vert d u \Vert^2 + \min(1,\frac{1}{2 c_p^2}) \Vert u \Vert^2 + \Vert p\Vert^2.
\end{aligned}
\end{equation}
Moreover we can bound the $H \Lambda$-norm of $(v,q)$ by the $H \Lambda$-norm of $(u,p)$ with a Poincaré inequality:
\begin{equation}
\begin{aligned}
\Vert v \Vert^2 + \Vert d v \Vert^2 + \Vert q \Vert^2 
&= \Vert d \delta n\Vert^2 + \Vert \delta s \Vert^2 + \Vert p \Vert^2 + \Vert d \delta s \Vert^2 + \Vert o \Vert^2 \\
&\leq \Vert d u \Vert^2 + c_p^2 \Vert d m \Vert^2 + \Vert p \Vert^2 + \Vert d m \Vert^2 + \Vert o \Vert^2 \\
&\leq (1+ c_p^2) \Vert u \Vert^2 + \Vert d u \Vert^2 + \Vert p \Vert^2.
\end{aligned}
\end{equation}
By the symmetry of the formulation this is enough to conclude with the \\
Babuška–Lax–Milgram theorem.
\end{proof}

Finally, let us translate this problem into the language of vector calculus in $\mathbb{R}^3$ (the $\mathbb{R}^2$ case being analogous  with the  appropriate definition of scalar and vector curl).
First of all the unknowns are in fact a $4$-tuple of fields, so we will write $u = (u_0, u_1, u_2, u_3)$, the subscript pertaining to the degree of the corresponding differential form.
We will keep the notation $\mathfrak{H}$ for the space of harmonic forms which  is simply a vector space of dimension $1$ containing the constant functions.
The problem is then written:
Given $(f_0, f_1, f_2, f_3) \in L^2(\Omega) \times (L^2(\Omega))^3 \times (L^2(\Omega))^3 \times L^2( \Omega)$,\\
find $u_0 \in H^1(\Omega)$, $u_1 \in H(\text{curl},\Omega)$, $u_2 \in H(\text{div}, \Omega)$, $u_3 \in L^2( \Omega)$, $p \in \mathfrak{H}$ 
such that $\forall v_0 \in H^1(\Omega)$, $\forall v_1 \in H(\text{curl},\Omega)$, $\forall v_2 \in H(\text{div}, \Omega)$, $\forall v_3 \in L^2( \Omega)$, $\forall q \in \mathfrak{H}$,
\begin{equation} \label{eq:weak4}
\begin{aligned}
\langle u_1, \nabla v_0 \rangle + \langle p, v_0 \rangle &= \langle f_0 , v_0 \rangle, \\
\langle u_2, \nabla \times v_1 \rangle + \langle \nabla u_0, v_1 \rangle &= \langle f_1, v_1 \rangle, \\
\langle u_3, \nabla \cdot v_2 \rangle + \langle \nabla \times u_1, v_2 \rangle &= \langle f_2, v_2 \rangle, \\
\langle \nabla \cdot u_2, v_3 \rangle &= \langle f_3, v_3 \rangle, \\
\langle u_0, q \rangle &= 0.
\end{aligned}
\end{equation}
This weak formulation was introduced  in \cite{Leopardi2016}, where its wellposedness is also established, thanks to the standard finite element exterior calculus tools \cite{feec-cbms}.
Natural boundary condition result from the weak formulation \eqref{eq:weak4} hence a solution $(u_0,u_1,u_2,u_3)$ satisfies $u_1 \in H_0( \text{div}, \Omega)$, $u_2 \in H_0( \text{curl}, \Omega)$ and $u_3 \in H^1_0(\Omega)$.
Thus a solution of \eqref{eq:weak4} satisfies
\begin{equation}
\begin{aligned}
- \nabla \cdot u_1 + p &= f_0, \\
\nabla u_0 + \nabla \times u_2 &= f_1, \\
\nabla \times u_1 - \nabla u_3 &= f_2, \\
\nabla \cdot u_2 &= f_3.
\end{aligned}
\end{equation}
with the natural boundary conditions
\begin{equation}
\begin{aligned}
 u_1 \cdot n&= 0 & \;\mathrm{on} \; \partial \Omega \\
u_2 \times n &= 0& \;\mathrm{on} \; \partial \Omega\\
u_3 &=0 & \;\mathrm{on} \; \partial \Omega 
\end{aligned}
\end{equation}

Solving the problem \eqref{eq:weak4} simultaneously solves two div-curl problems \eqref{eq:divcurlpb}, computes two Helmholtz decompositions (of $f_1$ and $f_2$) 
and finds two functions with prescribed gradient.
Assuming that the $f_i$-functions are compatible (for the sake of simplicity only, otherwise the problem is solved for their orthogonal projection in the appropriate spaces)
the problem div-curl \eqref{eq:divcurlpb} solved are:
\[ \begin{matrix}\nabla \cdot u_1 = -f_0 \\ \nabla \times u_1 = f_2 \\ u_1 \cdot n =0\; \mathrm{on} \;\partial\Omega\;\end{matrix}\quad \text{and}\quad \begin{matrix}\nabla \cdot u_2 = f_3 \\ \nabla \times u_2 = f_1\\ u_2\times n =0 \; \mathrm{on} \;\partial\Omega \end{matrix}\ .\]
\begin{remark} \label{rem:diffsol}
There are two differences between the two problems, first $u_1$ and $u_2$ do not follow the same boundary conditions, moreover
as we  shall  see below in section \ref{Finiteelements},  the discretization of the problem does not treat the differential  and  codifferential symmetrically.
In particular, we have  no error estimates for the convergence of the discrete codifferential  $|| \delta u - \delta_h u_h ||_{L^2}$, see section \ref{numerical}. %namely one of the two operators will be privileged in the sense that its discrete counterpart will converge faster to the continuous solution. 
\end{remark}

\section{Finite elements} \label{Finiteelements}
The design for finite elements suitable for exterior calculus saw substantial progress recently, with the seminal work of \cite{Arnold_2010},\cite{feec-cbms}.
For the Hodge-Dirac problem (as well as for the Hodge-Laplacian problem), these elements can be realized in a very generic way as shown in the periodic table of the finite elements \cite{periodictable}.

The main properties of these elements are that they form a discrete subcomplex and admit bounded cochain projections.
Being a discrete subcomplex means that the functions constructed on these elements belong to the domain of the exterior derivative 
(just as the functions are respectively included in $H^1$, $H(\text{curl})$, $H(\text{div})$, $L^2$)
and that their derivative (respectively their gradient, curl and div) are included in the functions of the next element. 
This last property allows to have properties exactly verified at the discrete level (for example  discrete fields with exactly zero divergence).
Bounded cochain projections are projections from continuous space to discrete space commuting with the exterior derivative. 
The boundedness for different norms ensures stability and  accurate estimation of the error. These projections exist mainly as theoretical tools 
and their calculation is never performed in the numerical scheme.

\begin{remark}
Although the discrete exterior derivative operator $d_h$ is the same as the continuous $d$ (more precisely its restriction on the discrete spaces), 
the discrete codifferential operator $\delta_h$ has little to do with the continuous operator $\delta$.
Indeed, it is the adjoint of the same operator but on a  different space.
This is why we have removed any occurrence of $\delta$ from the formulation \eqref{eq:weak1}.
\end{remark}

The space of harmonic forms $\mathfrak{H}$ remains in this case the space of constant functions.
Its determination can however be more complicated when using general domains or other boundary conditions. This problem is detailed in  section \ref{Boundaryconditions}.
\begin{remark}
The discrete spaces are then subspaces of continuous spaces and we use a conforming method.
This will not always be the case for general domains where the space of discrete harmonic forms may be different from the continuous space, though they have the same dimensions, see \cite{feec-cbms}.
\end{remark}

Although  other types of meshes such as quadrilaterals are possible \cite{periodictable}, we focus on simplicial meshes.
For each degree of forms, there are two families of piecewise polynomial elements indexed by their polynomial degree.
Consider a simplicial mesh $\mathfrak{T}$ (and denote the cells of the mesh by $T \in \mathfrak{T}$) and a polynomial degree $r$.
The first family is the complete space of polynomials, differing by the type of continuity desired at the interfaces, it is given by 
\begin{equation} \label{eq:fullpoly}
\begin{aligned}
P_r \Lambda^0(\mathfrak{T}) = \lbrace \omega \in H^1(\Omega), \;\forall T \in \mathfrak{T}, \omega_{\vert T} \in P_r(T,\mathbb{R}) \rbrace, \\
P_r \Lambda^1(\mathfrak{T}) = \lbrace \omega \in H(\text{curl}, \Omega), \;\forall T \in \mathfrak{T}, \omega_{\vert T} \in P_r(T,\mathbb{R}^3) \rbrace, \\
P_r \Lambda^2(\mathfrak{T}) = \lbrace \omega \in H(\text{div}, \Omega), \;\forall T \in \mathfrak{T}, \omega_{\vert T} \in P_r(T,\mathbb{R}^3) \rbrace, \\
P_r \Lambda^3(\mathfrak{T}) = \lbrace \omega \in L^2(\Omega), \;\forall T \in \mathfrak{T}, \omega_{\vert T} \in P_r(T,\mathbb{R}) \rbrace.
\end{aligned}
\end{equation}
The space  $P_r(T, \mathbb{R}^k)$ denotes the set of polynomials of degree $r$, defined on the domain $T$ with value in $\mathbb{R}^k$.
The second family are the so-called \emph{trimmed} space with less degrees of freedom, it follows the same continuity conditions at the interfaces as the first one but uses only a subset of the polynomials,
the exact definition of this set can be found in \cite{feec-cbms}, we will simply denote them here by $P_r^- \Lambda^k$. This second family is then given by 
\begin{equation} \label{eq:trimmedpoly}
\begin{aligned}
P_r^- \Lambda^0(\mathfrak{T}) = \lbrace \omega \in H^1(\Omega),\;\forall T \in \mathfrak{T}, \omega_{\vert T} \in P_r^- \Lambda^0 \rbrace, \\
P_r^- \Lambda^1(\mathfrak{T}) = \lbrace \omega \in H(\text{curl}, \Omega),\;\forall T \in \mathfrak{T}, \omega_{\vert T} \in P_r^- \Lambda^1 \rbrace, \\
P_r^- \Lambda^2(\mathfrak{T}) = \lbrace \omega \in H(\text{div}, \Omega),\;\forall T \in \mathfrak{T}, \omega_{\vert T} \in P_r^- \Lambda^2 \rbrace, \\
P_r^- \Lambda^3(\mathfrak{T}) = \lbrace \omega \in L^2(\Omega),\;\forall T \in \mathfrak{T}, \omega_{\vert T} \in P_r^- \Lambda^3 \rbrace.
\end{aligned}
\end{equation}

We always have $P_r^- \Lambda^0(\mathfrak{T}) = P_r \Lambda^0(\mathfrak{T})$ and $P_r^- \Lambda^3(\mathfrak{T}) = P_{r-1} \Lambda^3(\mathfrak{T})$.
The common names (as referred to in the table \cite{periodictable}) for these elements  in $\mathbb{R}^3$ are 
\begin{itemize}
\item Lagrange elements of degree $r$  for $P_r^- \Lambda^0(\mathfrak{T})$,
\item Nedelec's face elements of the first kind for $P_r^- \Lambda^1(\mathfrak{T})$ and of the second kind for $P_r \Lambda^1(\mathfrak{T})$, 
\item Nedelec's edge elements of the first kind for $P_r^- \Lambda^2(\mathfrak{T})$ and of the second kind for $P_r \Lambda^2(\mathfrak{T})$,
\item discontinuous Galerkin for $P_r \Lambda^3(\mathfrak{T})$.
\end{itemize}

From these elements, we have to build a sequence of degree increasing forms. 
For each degree, we can take either the full polynomial element or the trimmed polynomial element.
However, we must choose the appropriate polynomial degree, following a simple rule:
if the next element is a complete polynomial we must go down one degree, if it is trimmed we keep the same polynomial degree. 
Thus for any degree $r$ the sequences
\begin{equation} \label{eq:seq1}
P_r^- \Lambda^0(\mathfrak{T}) \rightarrow P_r^- \Lambda^1(\mathfrak{T}) \rightarrow P_r^- \Lambda^2(\mathfrak{T}) \rightarrow P_r^- \Lambda^3(\mathfrak{T}),
\end{equation}
\begin{equation} \label{eq:seq2}
P_{r+3} \Lambda^0(\mathfrak{T}) \rightarrow P_{r+2} \Lambda^1(\mathfrak{T}) \rightarrow P_{r+1} \Lambda^2(\mathfrak{T}) \rightarrow P_r \Lambda^3(\mathfrak{T})
\end{equation}
are both correct. 

The main difference between the two families comes from their approximation properties.
We define the approximation error for a discrete space $V_h$ embedded in a continuous space $V$ and a function $u \in V$ by $E(u) = \inf_{v_h \in V_h} \Vert u - v_h \Vert$.
Error estimates for schemes are often expressed in terms of these approximation errors.
These errors depend on the size $h$ of the cells in the mesh, and converge to $0$ when $h$ tends to $0$.
For these spaces we have in all cases an error estimate of the form $E(u) \leq C\, h^l \Vert u \Vert_{H^l}$ with $C$ a constant independent of $h$ and $l$ the order of convergence.
For complete spaces of degree $r$ the order of convergence of the approximation error (of a sufficiently regular function) is $r+1$ and the order of convergence of its derivative is $r$.
For trimmed spaces of degree $r$ the order of convergence is $r$ and the order of convergence of its derivative is also $r$.

\begin{remark} \label{rem:diffsol2}
To complete remark \ref{rem:diffsol}, we can see that in the discrete case another difference appears between the $1$-forms and the $2$-forms regarding the regularity of the solutions. According to whether we use Raviart-Thomas-Nedelec edge  (resp.  face)  elements  the discrete solution $u_h$ belongs  to $H(\text{div},\Omega)$  (resp.  $H(\text{curl},\Omega)$).
In the continuous case,  the solution $u$  belongs  to  both $H(\text{div},\Omega) \cap H(\text{curl},\Omega)$. 
\end{remark}

Let $V^0_h \rightarrow V^1_h \rightarrow V^2_h \rightarrow V^3_h$ be the chosen discrete space sequence and $\mathfrak{H}_h$ be the space of discrete harmonic forms,
the discrete problem is then:\\
Given $(f_0, f_1, f_2, f_3) \in L^2(\Omega) \times (L^2(\Omega))^3 \times (L^2(\Omega))^3 \times L^2( \Omega)$ 
find $u_{h0} \in V^0_h$, $u_{h1} \in V^1_h$, $u_{h2} \in V^2_h$, $u_{h3} \in V^3_h$, $p_h \in \mathfrak{H}_h$ 
such that $\forall v_{h0} \in V^0_h$, $\forall v_{h1} \in V^1_h$, $\forall v_{h2} \in V^2_h$, $\forall v_{h3} \in V^3_h$, $\forall q_h \in \mathfrak{H}_h$,
\begin{equation} \label{eq:weak4h}
\begin{aligned}
\langle u_{h1}, \nabla v_{h0} \rangle + \langle p_h, v_{h0} \rangle &= \langle f_0 , v_{h0} \rangle, \\
\langle u_{h2}, \nabla \times v_{h1} \rangle + \langle \nabla u_{h0}, v_{h1} \rangle &= \langle f_1, v_{h1} \rangle, \\
\langle u_{h3}, \nabla \cdot v_{h2} \rangle + \langle \nabla \times u_{h1}, v_{h2} \rangle &= \langle f_2, v_{h2} \rangle, \\
\langle \nabla \cdot u_{h2}, v_{h3} \rangle &= \langle f_3, v_{h3} \rangle, \\
\langle u_{h0}, q_h \rangle &= 0.
\end{aligned}
\end{equation}

The error estimates are calculated in  \cite{Leopardi2016}.
Let $K$ be the solution operator which takes $f \rightarrow u$ in \eqref{eq:weak4} and $\pi_h$ be the bounded cochain projection mentioned at the beginning of section \ref{Finiteelements}.
We define $\eta = \Vert (I - \pi_h)K \Vert$ and $\mu = \Vert (I -  \pi_h) P_\mathfrak{H} \Vert$.
In our case, both converge to $0$, and in practice we can use their expressions computed in \cite{Arnold_2010}, 
which for discrete spaces using polynomial degrees $r$ give $\eta = \mathcal{O}(h)$ and $\mu = \mathcal{O}(h^{r+1})$.
The error estimate for $u$ the continuous solution of \eqref{eq:weak4} and $u_h$ the discrete solution of \eqref{eq:weak4h} gives:
\begin{equation} \label{eq:errestimate}
\begin{aligned}
\Vert d(u - u_h) \Vert &\lesssim E(du), \\
\Vert u - u_h \Vert &\lesssim E(u) + \eta (E(du) + E(p)), \\
\Vert p - p_h \Vert &\lesssim E(p) + \mu E(du).
\end{aligned}
\end{equation}
Where the notation $A \lesssim B$ means that there is a constant $C$ independent of $u$ and $h$ such that $A \leq C B$.
The equations \eqref{eq:errestimate} include all the components of $u = (u_0,u_1,u_2,u_3)$.
\section{Implementation}\label{implem}
With these elements, the implementation of the problem \eqref{eq:weak4h} in \textit{unified form language (UFL)}, see \cite{Alnaes}, is straightforward:
\begin{lstlisting}[caption=Implementation of the variational formulation in ufl., label=code:3D]
degree = 2
elemf0 = FiniteElement('P', tetrahedron, degree)
elemf1 = FiniteElement('N1E', tetrahedron, degree)
elemf2 = FiniteElement('N1F', tetrahedron, degree)
elemf3 = FiniteElement('DG', tetrahedron, degree-1)
elemH = FiniteElement('Real', tetrahedron, 0)
W = MixedElement([elemf0,elemf1,elemf2,elemf3,elemH])

(u0,u1,u2,u3,uh) = TrialFunctions(W)
(v0,v1,v2,v3,vh) = TestFunctions(W)
a1 = (dot(grad(u0),v1) + dot(curl(u1),v2) + div(u2)*v3)*dx
a2 = (dot(u1,grad(v0)) + dot(u2,curl(v1)) + u3*div(v2))*dx
ah = uh*v0*dx + u0*vh*dx 
a = a1 + a2 + ah 
\end{lstlisting}
Our  FEniCS codes and tests are available through the GitHub repository \url{https://github.com/mlhanot/divcurl_solver}.

Being a mixed finite element scheme, the assembled linear system is semidefinite, which for very large number of degrees of freedom could be an issue.
However  for contractible domains we obtained reliable results  on a standard computer  (16 GB RAM)  with direct solvers up  to  several millions of degrees of freedom. For non contractible domains, the main workload is generated by the computations of a linear basis of the harmonic forms. Indeed this amounts to  computing a basis of the  null space of a matrix, which is trickier than  solving a regular linear system.
\section{Problem in two dimensions}\label{2D}
The exterior calculus framework is unified w.r.t.  dimensions, hence
the formulation \eqref{eq:weak31}-\eqref{eq:weak32} remains perfectly valid without any change (except that one needs $\Omega$ to be domain of $\mathbb{R}^2$ instead of $\mathbb{R}^3$).

The expression in the vector calculus formalism is however quite different, in particular there are two possibilities of identification.
The two corresponding problems are then (respectively for the curl identification and the div identification):\\
Given $(f_0,f_1,f_2) \in L^2(\Omega) \times (L^2(\Omega))^2 \times L^2(\Omega)$, find $u_0 \in H^1(\Omega)$, $u_1 \in H(\text{curl},\Omega)$, $u_2 \in L^2(\Omega)$, $p \in \mathfrak{H}$ 
such that $\forall v_0 \in H^1(\Omega)$, $\forall v_1 \in H(\text{curl},\Omega)$, $\forall v_2 \in L^2(\Omega)$, $\forall q \in \mathfrak{H}$,
\begin{equation} \label{eq:weak2d1}
\begin{aligned}
\langle u_1, \nabla v_0 \rangle + \langle p, v_0 \rangle &= \langle f_0, v_0 \rangle, \\
\langle u_2, \nabla \times v_1 \rangle + \langle \nabla u_0, v_1 \rangle &= \langle f_1, v_1 \rangle, \\
\langle \nabla \times u_1, v_2 \rangle &= \langle f_2 , v_2 \rangle, \\
\langle u_0, q \rangle &= 0.
\end{aligned}
\end{equation}
Given $(f_0,f_1,f_2) \in L^2(\Omega) \times (L^2(\Omega))^2 \times L^2(\Omega)$, find $u_0 \in H^1(\Omega)$, $u_1 \in H(\text{div},\Omega)$, $u_2 \in L^2(\Omega)$, $p \in \mathfrak{H}$ 
such that $\forall v_0 \in H^1(\Omega)$, $\forall v_1 \in H(\text{div},\Omega)$, $\forall v_2 \in L^2(\Omega)$, $\forall q \in \mathfrak{H}$,
\begin{equation} \label{eq:weak2d2}
\begin{aligned}
\langle u_1, \nabla^\perp v_0 \rangle + \langle p, v_0 \rangle &= \langle f_0, v_0 \rangle, \\
\langle u_2, \nabla \cdot v_1 \rangle + \langle \nabla^\perp u_0, v_1 \rangle &= \langle f_1, v_1 \rangle, \\
\langle \nabla \cdot u_1, v_2 \rangle &= \langle f_2 , v_2 \rangle, \\
\langle u_0, q \rangle &= 0.
\end{aligned}
\end{equation}

The difference between the two is then the same as in the remark \ref{rem:diffsol}, namely the boundary conditions and the preferred operator.

The choice of finite elements must also be appropriate, although similar to the $3$-dimensional case, one must be careful about the vector element chosen.
The appropriate elements are (using the same semantics as for \eqref{eq:fullpoly} and \eqref{eq:trimmedpoly}):
\begin{itemize}
\item Lagrange elements for $P_r^- \Lambda^0(\mathfrak{T})$,
\item Raviart-Thomas face (or edge) elements for $P_r^- \Lambda^1(\mathfrak{T})$,
\item Brezzi-Doublas-Marini face (or edge) elements for $P_r \Lambda^1(\mathfrak{T})$,
\item discontinuous Galerkin for $P_r \Lambda^2(\mathfrak{T})$.
\end{itemize}
We must use edge elements for the problem \eqref{eq:weak2d1} and face elements for the problem \eqref{eq:weak2d2}.

\begin{remark}
The problem \eqref{eq:weak2d2} is in fact simply the problem \eqref{eq:weak2d1} to which we apply a quarter rotation on the vector space.
This means that we solve the problem in the space of differential forms and only then apply the identification of Fig. \ref{cd:dim2} to come back to the vector fields.
\end{remark}

\section{Non contractible domain and harmonic forms} \label{Noncontractible}
When the domain is no longer contractible, the Hodge decomposition as given in \eqref{eq:hodge} is no longer valid.
As described in section \ref{Exteriorcalculus} the problem comes from the appearing of harmonic forms, i.e. elements in the kernel of the Hodge-Dirac operator 
(equivalently in the kernel of the Hodge-Laplacian or fields $f$ such that $\nabla \cdot f = 0$ and $\nabla \times f = 0$).

This is not really an issue  because we already had harmonic forms in \eqref{eq:weak1} and we can treat them in the same way.
However this poses the problem of determining the space of harmonic forms %(as well as the possible risk of losing a certain compactness property necessary for the proof).
The theoretical aspect of the problem is solved thanks to the famous theorem of De Rham giving the isomorphism between the harmonic forms and the cohomology of the cochain complex.
This cohomology has a very strong geometrical interpretation, its dimension is given by the Betti numbers.
Thus in dimension $2$ the number of harmonic $0$-forms is the number of connected components of the domain, the number of harmonic $1$-forms is the number of holes of the domain and there are no harmonic $2$-forms.
In dimension $3$ the number of harmonic $0$-forms is still the number of connected components, the number of harmonic $1$-forms corresponds to the number of tunnels and the number of harmonic $2$-forms corresponds to the 
number of vacuum bubbles, there is no harmonic $3$-forms.
This is illustrated in figure \ref{fig:crossH1} showing the two harmonic $1$-forms on a disk with two holes,
and figures \ref{fig:crossH2} and \ref{fig:crossH3} showing the two harmonic $1$-forms on a hollow torus.
\begin{figure}
\centering
\includegraphics[width=0.6\textwidth]{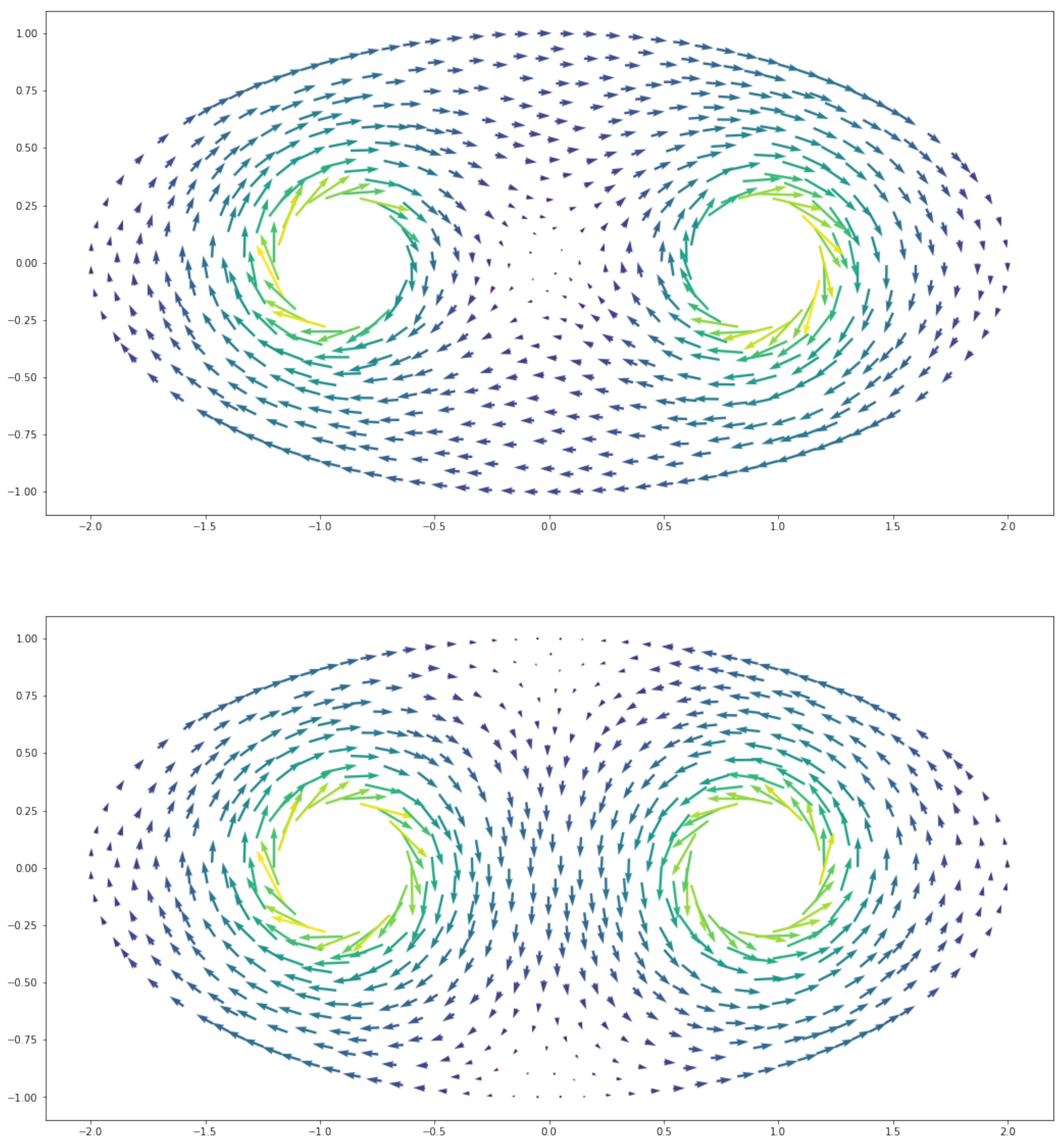}
\caption{Harmonic $1$-forms on a surface with two holes.}
\label{fig:crossH1}
\end{figure}

\begin{figure}
\centering
\includegraphics[width=\textwidth]{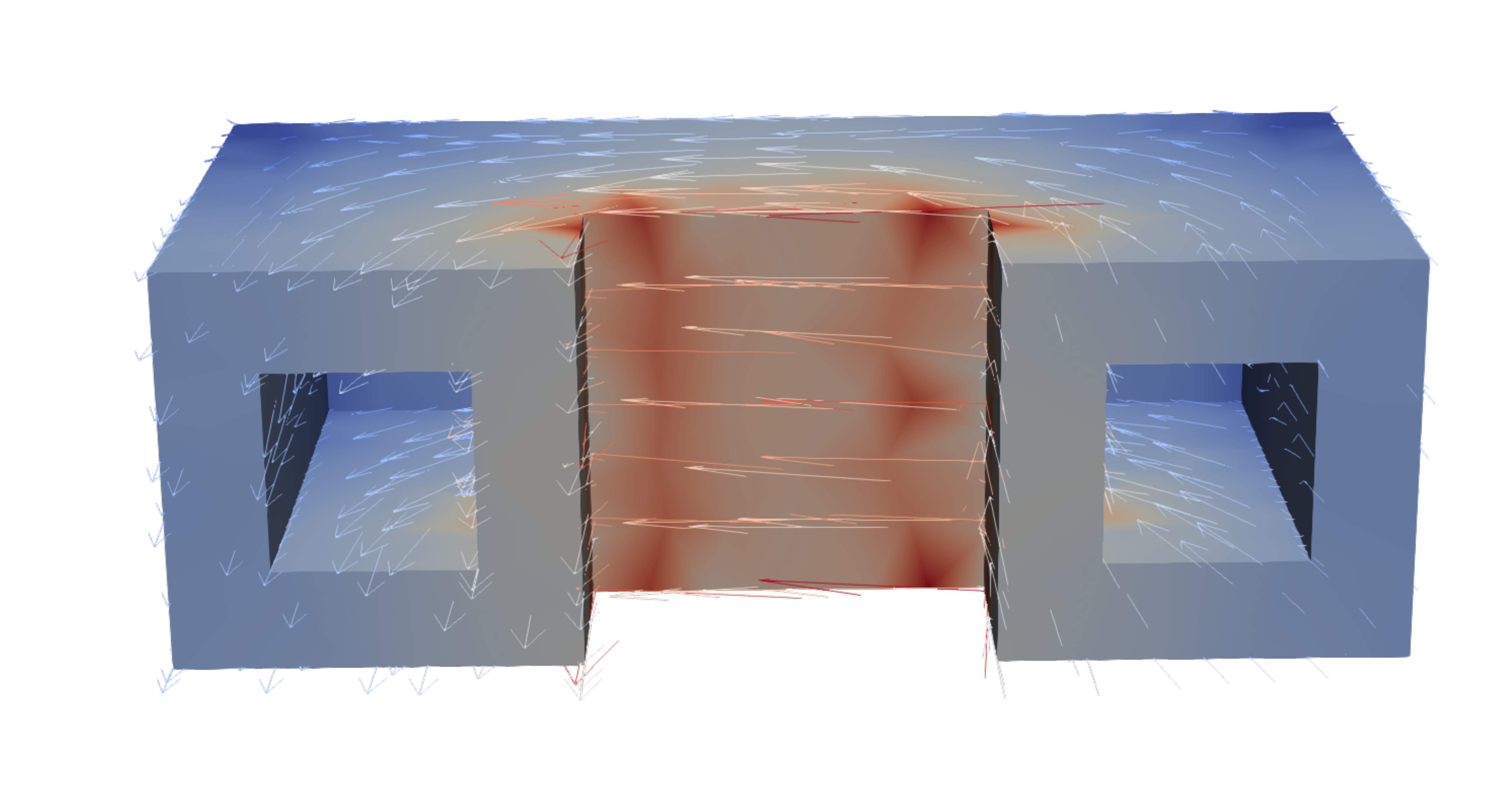}
\caption{First harmonic $1$-form on a 3D hollow torus sliced for visualisation.}
\label{fig:crossH2}
\end{figure}

\begin{figure}
\centering
\includegraphics[width=\textwidth]{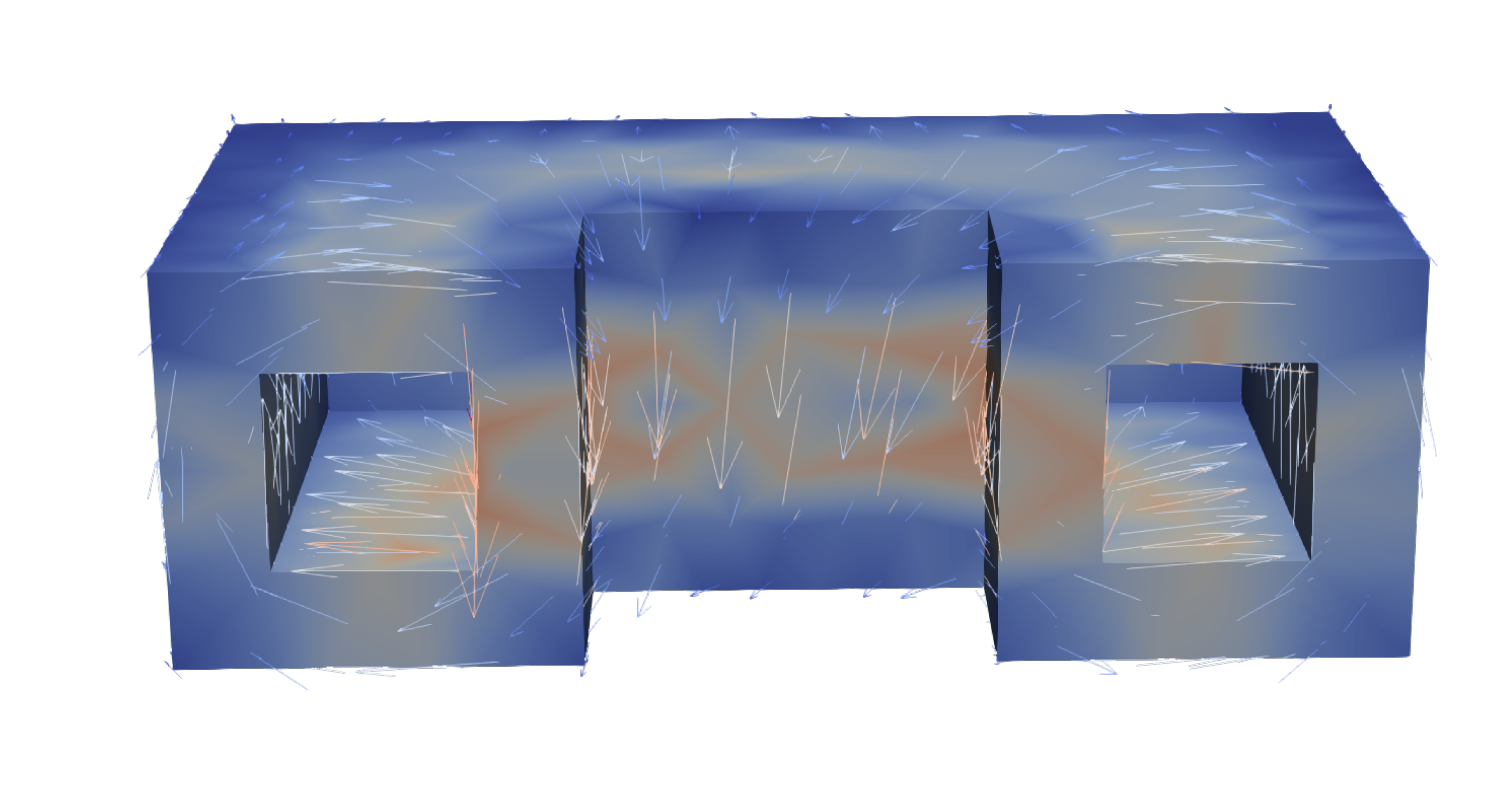}
\caption{Second harmonic $1$-form on a 3D hollow torus sliced for visualisation.}
\label{fig:crossH3}
\end{figure}
Another important theorem gives the isomorphism between discrete and continuous harmonic forms (see \cite{feec-cbms}, chapter 5 and 7.6), 
so the dimension  of  the space of harmonic forms does not  depend on the discretization or the elements chosen.

However the actual search for these forms is  much more complicated. 
The isomorphisms only give their dimension number, and a heuristic idea of their shapes.
To compute them, we will start from their definition as kernel of the Hodge-Dirac operator, which corresponds exactly to the kernel of the assembled matrix of the system (see \cite{Arnold_2010}).
The problem of determining a basis of harmonic forms becomes in practice a problem of finding kernels of matrices.
The dimension of these kernels is a useful information for many algorithms (for example to search for the smallest eigenvalues) and the idea of the solution  shapes can give a good initial guess.
However, the problem remains intricate because of the size of the linear  systems.

In our numerical computations with FEniCS, we have achieved the best results with the  numerical algebra library SLEPc, see  \url{https://slepc.upv.es/}.

Once a basis of harmonic forms is determined, it is enough to add them to the space $\mathfrak{H}$ in the scheme. 
The proofs are done in this general framework and still give the right estimates.

\section{Boundary conditions} \label{Boundaryconditions}
So far, we have not imposed any essential (Dirichlet) conditions on our spaces, so natural conditions have emerged.
Although these conditions are sufficient to ensure wellposedness, they fix the degree of forms in which we look for our solution 
and we have noticed in the remarks \ref{rem:diffsol} and \ref{rem:diffsol2} that these choices have an impact on the convergence and the regularity of our solutions. Specifically, natural boundary conditions  $u \times n =0$  (resp. $u \cdot n =0$ ) imposes $u$ to be a 2-form, i.e. $u \in  H(\text{div},\Omega)$  
(resp.  $u$ to be a 1-form, i.e. $H(\text{curl},\Omega)$). 
If we  wish that $u_h$ belongs to the  space   which does not correspond to the natural condition obtained, a simple way is
 to apply homogeneous Dirichlet conditions to all spaces. 
The sequence then becomes 
\begin{equation}
\begin{tikzcd}
H_0^1(\Omega) \arrow[r, "\nabla"] & H_0(\text{curl},\Omega) \arrow[r, "\nabla \times"] & H_0(\text{div},\Omega) \arrow[r, "\nabla \cdot"] & L^2(\Omega).
\end{tikzcd}
\end{equation}

From a theoretical point of view, enforcing Dirichlet conditions on all spaces does not pose any problem, 
we just end up with another complex, dual of the first one and all the theorems still work with one difference:
the number (dimension) of harmonic forms are inverted, so under these conditions there are no $0$ harmonic forms and as many $3$ harmonic forms (in dimension $3$) as there are connected components.
One must adjust the $\mathfrak{H}$ space in the formulation, so replace $\langle p, v_0 \rangle$ by $\langle p, v_3 \rangle$ in \eqref{eq:weak4}.

The situation becomes more complicated when mixed conditions are applied (natural on some faces, essential on others).
New harmonic forms can then appear even for simple domains
as illustrated in figure \ref{fig:Hrect}.
\begin{figure}
\center
\includegraphics{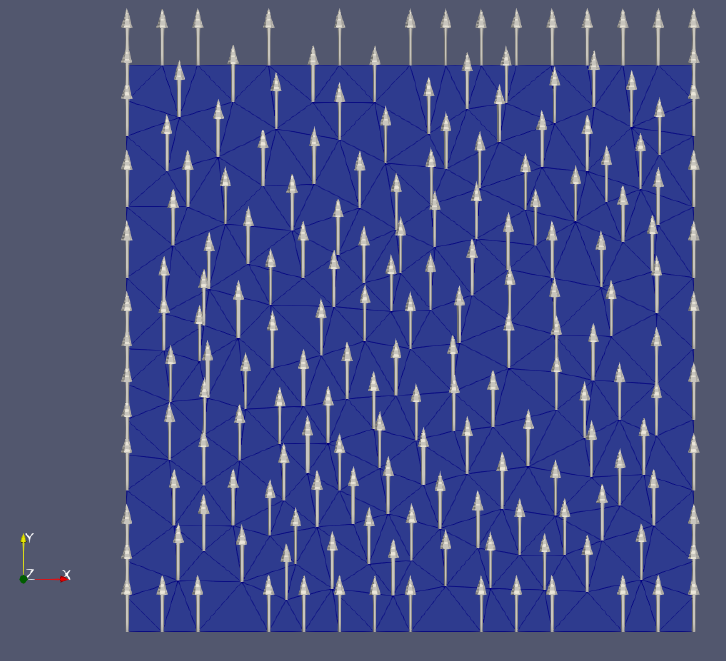}
\caption{Harmonic $1$-form on a contractible domain with mixed boundary conditions.}
\label{fig:Hrect}
\end{figure}

\section{Numerical application}\label{numerical}
To conclude we present some results computed with the weak form \eqref{eq:weak4h}.
In dimension 2 we take for reference function \eqref{eq:refd2} on the unit square $(0,1)^2$, 
tailored to accommodate the boundary conditions, to have a non-trivial divergence and curl without being symmetric.
In dimension 3 we take for reference function \eqref{eq:refd3} on the unit cube $(0,1)^3$.
The test with non-trivial harmonic forms is delicate, indeed the reference function will not be orthogonal to these forms and these forms are not known for most of the meshes. Thus  we computed the rate of convergence in dimension 2 only, 
using periodic conditions  on the  edges $(x=0)$ and $(x=1)$ of the unit square with vanishing tangential components on the other edges  because the harmonic forms are explicitly known in this case, see Fig. \ref{fig:Hrect}.
\begin{equation} \label{eq:refd2}
u = \begin{pmatrix} \sin(3 \pi x)\cos(\pi y)\\ 
\sin(\pi y) \cos(2 \pi x)
\end{pmatrix}
\end{equation}

\begin{equation} \label{eq:refd3}
u = \begin{pmatrix} \sin(3 \pi x) \cos(\pi y) z\\ 
\sin(\pi y) \cos(2 \pi x)+z \\
\sin(\pi z) \cos(3 \pi x) \cos(\pi y)
\end{pmatrix}
\end{equation}

We report the rate of convergence of the discrete solution towards the analytical solution, 
the rate of convergence of its differential (divergence or curl as the case may be) and the rate of convergence of its codifferential.
The latter is calculated in two steps, first we make the orthogonal projection $P$ of the function on an appropriate space 
(if the function has been defined on the face elements we project it on the corresponding edge elements and inversely).
Then we compute its codifferential.
Although there is no theoretical result on this convergence to our knowledge,  
we can observe in the last columns of the tables a convergence  rate  for the codifferential but with a loss of one degree with respect to the approximation property of the space of $u_h$. Of course as reported in   tables \ref{table:dim2Adjdeg1}, \ref{table:dim2Directdeg1}, \ref{table:dim31formdeg1} and \ref{table:dim32formdeg1},  when degree one polynomials are used  no convergence of the codifferentials can be expected.

Using a contractible domain and the sequence of finite elements \eqref{eq:seq1}, we can see the convergence rates in dimension $2$ for the divergence identification 
for  polynomials of degree $1$ and $2$ in  tables \ref{table:dim2Adjdeg1} and \ref{table:dim2Adjdeg2}, 
for the curl identification in  tables \ref{table:dim2Directdeg1} and \ref{table:dim2Directdeg2}.
In dimension $3$ we can see the convergence rates for  $1$-forms in  tables \ref{table:dim31formdeg1} and \ref{table:dim31formdeg2},
for  $2$-forms in tables \ref{table:dim32formdeg1} and \ref{table:dim32formdeg2}.
We can see the convergence rates with the sequence of elements \eqref{eq:seq2} of degree $r = 0$ and $r = 1$ in dimension $2$ with the divergence identification in  tables \ref{table:dim2Adjfull1} and \ref{table:dim2Adjfull2}.
Finally, we can see the rate of convergence when harmonic $1$-forms are present in dimension $2$ with the divergence identification using the sequence \eqref{eq:seq2} of polynomial degree $2$ in the table \ref{table:dim2AdjHdeg2}.
We notice in tables \ref{table:dim2Adjfull1} and \ref{table:dim2Adjfull2} that using \emph{different} polynomials degrees like in equation \eqref{eq:seq2}  we have no loss  of convergence  for  the codifferential. We have no theoretical proof of this fact.

\begin{table}
\centering
\begin{tabular}{|c|c|c|c|c|c|c|} \hline
 h & $\Vert u - u_h \Vert$ & rate & $\Vert \nabla \cdot (u - u_h) \Vert$ & rate & $\Vert \nabla \times (u - P u_h) \Vert$ & rate \tabularnewline
 \hline
0.1414 & 0.1753 & --- & 1.1268 & --- & 2.8449 & --- \\
\hline
0.0707 &0.0854 & 1.03 & 0.5660 &0.99 & 4.8353 &-0.76\\
\hline
0.0353 &0.0427 &0.99 & 0.2833 &0.99 & 4.7347 &0.03\\
\hline
0.0176 & 0.0213 &1.00 & 0.1416 &0.99 & 4.8947 &-0.04\\
\hline
0.0088 &0.0107 &0.99 & 0.0708 &0.99 & 4.3588 &0.16\\
\hline
0.0044 &0.0053 &0.99 & 0.0354 &0.99 & 4.0170 & 0.11\\
\hline
0.0022 &0.0026& 1.01& 0.0177 &0.99 & 4.7146 &-0.23\\
\hline
\end{tabular}
\caption{Convergence rates for the sequence \eqref{eq:seq1} of degree $1$ in $2$-dimension with the divergence identification.} \label{table:dim2Adjdeg1}
\end{table}

\begin{table}
\centering
\begin{tabular}{|c|c|c|c|c|c|c|} \hline
 h & $\Vert u - u_h \Vert$ & rate & $\Vert \nabla \cdot (u - u_h) \Vert$ & rate & $\Vert \nabla \times (u - P u_h) \Vert$ & rate \tabularnewline
 \hline
0.1414 &0.04518 & --- &0.43621 & --- &1.09441 & ---\\
\hline
0.0707 &0.01084 &2.05 &0.11165 &1.96 &0.50165 &1.12\\
\hline
0.0353 &0.00272 &1.99 &0.02806 &1.99 &0.25539 &0.97\\
\hline
0.0176 &0.00068 &1.99 &0.00702 &1.99 &0.12753 &1.00\\
\hline
0.0088 &0.00017 &1.99 &0.00175 &1.99 &0.06557 &0.95\\
\hline
0.0044 &0.00004 &1.99 &0.00043 &1.99 &0.03346 &0.97\\
\hline
\end{tabular}
\caption{Convergence rates for the sequence \eqref{eq:seq1} of degree $2$ in $2$-dimension with the divergence identification.} \label{table:dim2Adjdeg2}
\end{table}

\begin{table}
\centering
\begin{tabular}{|c|c|c|c|c|c|c|} \hline
 h & $\Vert u - u_h \Vert$ & rate & $\Vert \nabla \cdot (u - u_h) \Vert$ & rate & $\Vert \nabla \times (u - P u_h) \Vert$ & rate \tabularnewline
 \hline
0.1414 &0.027239 & --- &1.126863 & --- &0.909705 & ---\\
\hline
0.0707 &0.006622 &2.04 &0.566014 &0.99 &0.401096 &1.18\\
\hline
0.0353 &0.001653 &2.00 &0.283319 &0.99 &0.202041 &0.98\\
\hline
0.0176 &0.000413 &2.00 &0.141698 &0.99 &0.101697 &0.99\\
\hline
0.0088 &0.000103 &1.99 &0.070854 &0.99 &0.050725 &1.00\\
\hline
0.0044 &0.000026 &1.99 &0.035427 &0.99 &0.025886 &0.97\\
\hline
0.0022 &0.000006 &2.00 &0.017713 &0.99 &0.012923 &1.00\\
\hline
\end{tabular}
\caption{Convergence rates for the sequence \eqref{eq:seq2} of degree $2$-$1$-$0$ in $2$-dimension with the divergence identification.} \label{table:dim2Adjfull1}
\end{table}

\begin{table}
\centering
\begin{tabular}{|c|c|c|c|c|c|c|} \hline
 h & $\Vert u - u_h \Vert$ & rate & $\Vert \nabla \cdot (u - u_h) \Vert$ & rate & $\Vert \nabla \times (u - P u_h) \Vert$ & rate \tabularnewline
 \hline
0.1414 &0.041798 & --- &0.436218 & --- &0.092080 & ---\\
\hline
0.0707 &0.009947 &2.07 &0.111658 &1.96 &0.021575 &2.09\\
\hline
0.0353 &0.002500 &1.99 &0.028061 &1.99 &0.005424 &1.99\\
\hline
0.0176 &0.000626 &1.99 &0.007024 &1.99 &0.001352 &2.00\\
\hline
0.0088 &0.000156 &1.99 &0.001756 &1.99 &0.000344 &1.97\\
\hline
0.0044 &0.000039 &1.99 &0.000439 &1.99 &0.000087 &1.97\\
\hline
\end{tabular}
\caption{Convergence rates for the sequence \eqref{eq:seq2} of degree $3$-$2$-$1$ in $2$-dimension with the divergence identification.} \label{table:dim2Adjfull2}
\end{table}

\begin{table}
\centering
\begin{tabular}{|c|c|c|c|c|c|c|} \hline
 h & $\Vert u - u_h \Vert$ & rate & $\Vert \nabla \times (u - u_h) \Vert$ & rate & $\Vert \nabla \cdot (u - P u_h) \Vert$ & rate \tabularnewline
 \hline
0.1414 &0.1754 & --- &0.6344 & --- &3.9210 & ---\\
\hline
0.0707 &0.0832 &1.07 &0.3182 &0.99 &3.8695 &0.01\\
\hline
0.0353 &0.0414 &1.00 &0.1592 &0.99 &3.6632 &0.07\\
\hline
0.0176 &0.0207 &1.00 &0.0796 &0.99 &3.5675 &0.38\\
\hline
0.0088 &0.0104 &0.99 &0.0398 &0.99 &3.3669 &0.08\\
\hline
0.0044 &0.0052 &0.98 &0.0199 &0.99 &2.9584 &0.18\\
\hline
0.0022 &0.0026 &1.01 &0.0099 &0.99 &3.1698 &-0.09\\
\hline
\end{tabular}
\caption{Convergence rates for the sequence \eqref{eq:seq1} of degree $1$ in $2$-dimension with the curl identification.} \label{table:dim2Directdeg1}
\end{table}

\begin{table}
\centering
\begin{tabular}{|c|c|c|c|c|c|c|} \hline
 h & $\Vert u - u_h \Vert$ & rate & $\Vert \nabla \times (u - u_h) \Vert$ & rate & $\Vert \nabla \cdot (u - P u_h) \Vert$ & rate \tabularnewline
 \hline
0.1414 &0.029911 & --- &0.207629 & --- &1.375620 &---\\
\hline
0.0707 &0.007369 &2.02 &0.052777 &1.97 &0.671374 &1.03\\
\hline
0.0353 &0.001851 &1.99 &0.013249 &1.99 &0.342535 &0.97\\
\hline
0.0176 &0.000463 &1.99 &0.003315 &1.99 &0.171992 &0.99\\
\hline
0.0088 &0.000116 &1.99 &0.000829 &1.99 &0.086888 &0.98\\
\hline
0.0044 &0.000029 &1.99 &0.000207 &1.99 &0.044053 &0.97\\
\hline
\end{tabular}
\caption{Convergence rates for the sequence \eqref{eq:seq1} of degree $2$ in $2$-dimension with the curl identification.} \label{table:dim2Directdeg2}
\end{table}

\begin{table}
\centering
\begin{tabular}{|c|c|c|c|c|c|c|} \hline
 h & $\Vert u - u_h \Vert$ & rate & $\Vert \nabla \times (u - u_h) \Vert$ & rate & $\Vert \nabla \cdot (u - P u_h) \Vert$ & rate \tabularnewline
 \hline
0.3464 &0.3128 & --- &2.2284 & --- &4.1974 & ---\\
\hline
0.1732 &0.1590 &0.97 &1.1579 &0.94 &4.2920 &-0.03\\
\hline
0.0866 &0.0790 &1.00 &0.5852 &0.98 &4.1096 &0.03\\
\hline
0.0044 &0.0393 &1.00 &0.2934 &0.99 &3.9851 &0.04\\
\hline
\end{tabular}
\caption{Convergence rates for the sequence \eqref{eq:seq1} of degree $1$ for $1$-forms in $3$-dimension.} \label{table:dim31formdeg1}
\end{table}

\begin{table}
\centering
\begin{tabular}{|c|c|c|c|c|c|c|} \hline
 h & $\Vert u - u_h \Vert$ & rate & $\Vert \nabla \times (u - u_h) \Vert$ & rate & $\Vert \nabla \cdot (u - P u_h) \Vert$ & rate \tabularnewline
 \hline
0.3464 &0.0598 & --- &0.5617 & --- &1.9609 & ---\\
\hline
0.1732 &0.0151 &1.97 &0.1476 &1.92 &0.9729 &1.01\\
\hline
0.0866 &0.0038 &1.99 &0.0373 &1.98 &0.4924 &0.98\\
\hline
\end{tabular}
\caption{Convergence rates for the sequence \eqref{eq:seq1} of degree $2$ for $1$-forms in $3$-dimension.} \label{table:dim31formdeg2}
\end{table}

\begin{table}
\centering
\begin{tabular}{|c|c|c|c|c|c|c|} \hline
 h & $\Vert u - u_h \Vert$ & rate & $\Vert \nabla \cdot (u - u_h) \Vert$ & rate & $\Vert \nabla \times (u - P u_h) \Vert$ & rate \tabularnewline
 \hline
0.3464 &0.2949 & --- &1.0490 & --- &4.4544 & ---\\
\hline
0.1732 &0.1439 &1.03 &0.5307 &0.98 &3.8759 &0.20\\
\hline
0.0866 &0.0722 &0.99 &0.2663 &0.99 &3.9021 &-0.00\\
\hline
0.0044 &0.0361 &0.99 &0.1333 &0.99 &3.9106 &-0.00\\
\hline
\end{tabular}
\caption{Convergence rates for the sequence \eqref{eq:seq1} of degree $1$ for $2$-forms in $3$-dimension.} \label{table:dim32formdeg1}
\end{table}

\begin{table}
\centering
\begin{tabular}{|c|c|c|c|c|c|c|} \hline
 h & $\Vert u - u_h \Vert$ & rate & $\Vert \nabla \cdot (u - u_h) \Vert$ & rate & $\Vert \nabla \times (u - P u_h) \Vert$ & rate \tabularnewline
 \hline
0.3464 &0.1083 & --- &0.8580 & --- &1.9663 & ---\\
\hline
0.1732 &0.02649 &2.03 &0.2301 &1.89 &0.9586 &1.03\\
\hline
0.0866 &0.0067 &1.97 &0.0587 &1.97 &0.4971 &0.94\\
\hline
\end{tabular}
\caption{Convergence rates for the sequence \eqref{eq:seq1} of degree $2$ for $2$-forms in $3$-dimension.} \label{table:dim32formdeg2}
\end{table}

\begin{table}
\centering
\begin{tabular}{|c|c|c|c|c|c|c|} \hline
 h & $\Vert u - u_h \Vert$ & rate & $\Vert \nabla \cdot (u - u_h) \Vert$ & rate & $\Vert \nabla \times (u - P u_h) \Vert$ & rate \tabularnewline
 \hline
0.1414 &0.04518 & --- &0.43621 & --- &1.09441 &---\\
\hline
0.0707 &0.01084 &2.05 &0.11165 &1.96 &0.50165 &1.12\\
\hline
0.0353 &0.00272 &1.99 &0.02806 &1.99 &0.25539 &0.97\\
\hline
0.0176 &0.00068 &1.99 &0.00702 &1.99 &0.12753 &1.00\\
\hline
0.0088 &0.00017 &1.99 &0.00175 &1.99 &0.06557 &0.95\\
\hline
\end{tabular}
\caption{Convergence rates for the sequence \eqref{eq:seq1} of degree $2$ in $2$-dimension with the divergence identification on a non contractible domain.} \label{table:dim2AdjHdeg2}
\end{table}

\section*{Acknowledgements}The first author would like to thank the Isaac Newton Institute for Mathematical Sciences, Cambridge, for support and hospitality during the programme \emph{GCS-Geometry, compatibility and structure preservation in computational differential equations, 2019} where work on this paper was initiated. This work was supported by EPSRC grant no EP/R014604/1.

\printbibliography

\end{document}